\pdfoutput=1
\documentclass[11pt]{article}
\usepackage{xcolor}
\usepackage{graphicx}
\usepackage{amsmath,amsfonts,amssymb,graphics,amsthm}
\usepackage{hyperref}
\usepackage{comment}
\usepackage{tabularx}
\usepackage[protrusion=true,expansion=true]{microtype}
\usepackage{enumerate}
\usepackage{bbm}
\usepackage{mathrsfs}
\usepackage[margin=1in]{geometry}

\definecolor{antiquefuchsia}{rgb}{0.57, 0.36, 0.51}

\hypersetup{
    colorlinks=false,
    linktocpage,
    }

\numberwithin{equation}{section}

\newtheorem{theorem}{Theorem}[section]

\newtheorem{lemma}[theorem]{Lemma}
\newtheorem{proposition}[theorem]{Proposition}

\newtheorem{remark}[theorem]{Remark}
\newtheorem{definition}[theorem]{Definition}

\theoremstyle{remark}

\def\@rst #1 #2other{#1}
\newcommand\MR[1]{\relax\ifhmode\unskip\spacefactor3000 \space\fi
  \MRhref{\expandafter\@rst #1 other}{#1}}
\newcommand{\MRhref}[2]{\href{http://www.ams.org/mathscinet-getitem?mr=#1}{MR#2}}

\def\MR#1{\href{http://www.ams.org/mathscinet-getitem?mr=#1}{MR#1}}

\newcommand{\C}{\mathbbm{C}}

\newcommand{\E}{\mathbbm{E}}
\newcommand{\N}{\mathbbm{N}}

\newcommand{\Z}{\mathbbm{Z}}
\newcommand{\R}{\mathbbm{R}}
\renewcommand{\P}{\mathbbm{P}}

\newcommand{\bbS}{\mathbbm{S}}

\newcommand{\lattice}{\mathrm{lattice}}
\newcommand{\LFPP}{\mathrm{LFPP}}

\DeclareMathOperator{\dit}{dist}

\DeclareMathOperator{\Var}{Var}

\def\alb#1\ale{\begin{align*}#1\end{align*}}
\def\allb#1\alle{\begin{align}#1\end{align}}

\newcommand{\aryb}{\begin{eqnarray*}}
\newcommand{\arye}{\end{eqnarray*}}
\def\alb#1\ale{\begin{align*}#1\end{align*}}
\newcommand{\eqb}{\begin{equation}}
\newcommand{\eqe}{\end{equation}}
\newcommand{\eqbn}{\begin{equation*}}
\newcommand{\eqen}{\end{equation*}}

\newcommand{\wt}{\widetilde}
\newcommand{\wh}{\widehat}

\let\originalleft\left
\let\originalright\right
\renewcommand{\left}{\mathopen{}\mathclose\bgroup\originalleft}
\renewcommand{\right}{\aftergroup\egroup\originalright}

\DeclareMathAlphabet{\mathpzc}{OT1}{pzc}{m}{it}

\begin{document}

\title{Comparison of discrete and continuum Liouville first passage percolation}
\author{
Morris Ang\footnote{Department of Mathematics, Massachusetts Institute of Technology. Email: \href{mailto:angm@mit.edu}{angm@mit.edu}}
} 
\date{  }
\maketitle

\begin{abstract}
Discrete and continuum Liouville first passage percolation (DLFPP, LFPP) are two approximations of the conjectural $\gamma$-Liouville quantum gravity (LQG) metric, obtained by exponentiating the discrete Gaussian free field (GFF) and the circle average regularization of the continuum GFF respectively. We show that these two models can be coupled so that with high probability distances in these models agree up to $o(1)$ errors in the exponent, and thus have the same distance exponent.

Ding and Gwynne (2018) give a formula for the continuum LFPP distance exponent in terms of the $\gamma$-LQG dimension exponent $d_\gamma$. Using results of Ding and Li (2018) on the level set percolation of the discrete GFF, we bound the DLFPP distance exponent and hence obtain a new lower bound $d_\gamma \geq 2 + \frac{\gamma^2}2$. This improves on previous lower bounds for $d_\gamma$ for the regime $\gamma \in (\gamma_0, 0.576)$, for some small nonexplicit $\gamma_0 > 0$.
\end{abstract}


\section{Introduction}

\subsection{Overview}
Let $h$ be a continuum Gaussian free field (GFF) on a simply connected domain $D \subset \C$. For $\gamma \in (0,2]$, the \emph{$\gamma$-Liouville quantum gravity ($\gamma$-LQG)} surface is, heuristically speaking, the random two-dimensional Riemannian manifold with metric given by $e^{\gamma h}(dx^2 + dy^2)$. This definition does not make literal sense as $h$ is a distribution (and so cannot be evaluated pointwise), but by using regularization procedures one can make sense of the random volume form of $\gamma$-LQG \cite{kahane, shef-kpz, rhodes-vargas-review}. An important open problem is to understand the metric structure of $\gamma$-LQG. In the special case $\gamma = \sqrt{\frac83}$, it was shown in \cite{lqg-tbm1, lqg-tbm2, lqg-tbm3} that $\sqrt{\frac83}$-LQG admits a natural metric structure which is isometric to the \emph{Brownian map}, a random metric space that is the scaling limit of uniform random planar maps \cite{legall-uniqueness, miermont-brownian-map}. The construction of the $\sqrt{\frac83}$-LQG metric is via a continuum growth process, and depends crucially on properties unique to $\gamma=\sqrt{\frac83}$.

In contrast, in recent years there have been many works trying to understand the conjectural $\gamma$-LQG metric for general $\gamma$ via various discretizations of $\gamma$-LQG. The papers \cite{ghs-map-dist, dzz-heat-kernel, dg-lqg-dim} prove the existence of a universal exponent $d_\gamma$ that describes distances in many of these discretizations, including Liouville graph distance, random planar maps, constructions involving the Liouville heat kernel, and \emph{continuum Liouville first passage percolation (LFPP)} which we define as follows. For a GFF $h$ on $D$ and $\xi > 0$, the continuum LFPP distance is the distance with respect to the Riemannian metric $e^{\xi h_1(z)} (dx^2 + dy^2)$, where $h_1(z)$ denotes the average of $h$ over the radius 1 circle $\partial B_1(z)$. Then for $D = [0,n]^2$ and $\xi = \frac{\gamma}{d_\gamma}$ the continuum LFPP distances scale as $n^{\frac{2}{d_\gamma} + \frac{\gamma^2}{2d_\gamma} + o(1)}$ \cite[Theorem 1.5]{dg-lqg-dim}. Continuum LFPP has also been studied in other works\footnote{Other works define continuum LFPP slightly differently; see Remark~\ref{rem_LFPP_defn}.} \cite{ding-goswami-watabiki, df-lqg-metric, GP-LFPP,DDDF-tightness}.

A discrete analog of continuum LFPP is \emph{discrete Liouville first passage percolation} (DLFPP), in which one samples a discrete Gaussian free field (DGFF) $\eta$ on an $n \times n$ lattice, assigns a weight of $e^{\xi \sqrt{\frac\pi2}\eta(v)}$ to each vertex $v$, and defines the distance between two vertices to be the weight of the minimum-weight path between the vertices. Previous works have studied DLFPP distances \cite{ding-goswami-watabiki}, geodesics \cite{dingzhang-geodesic}, and subsequential scaling limits \cite{DD19}. 

In this work we show that with high probability, up to an $n^{o(1)}$ multiplicative error $\xi$-DLFPP distances agree with $\xi$-continuum LFPP distances, and thereby conclude that for $\xi = \frac{\gamma}{d_\gamma}$ the $\xi$-DLFPP distance exponent agrees with the $\xi$-LFPP distance exponent. This proves a conjecture of \cite[Section 1.5]{dg-lqg-dim}. We can then use existing results on DGFF level set percolation \cite{ding-li-chem-dist} to upper bound DLFPP annulus crossing distances, leading to a new lower bound $d_\gamma \geq 2 + \frac{\gamma^2}{2}$. This lower bound is the best known for the range $\gamma \in (\gamma_0,0.576)$, where $\gamma_0>0$ is small and non-explicit. 
\medskip

\noindent \textbf{Acknowledgements:} We thank Jian Ding, Ewain Gwynne, and Scott Sheffield for helpful discussions. We especially thank Ewain Gwynne for suggesting the problem and for valuable comments on an earlier draft.

\medskip

\subsection{Main results}

Let $\bbS = [0,1]^2$ be the unit square. For any point $z \in \R^2$, let $[z]$ denote the lattice point closest to $z$. For any set $A \subset \R^2$, let $[A] = \{ [a] \: : \: a \in A\}$ be its lattice approximation, and for any positive integer $n \in \N$ write $nA = \{ na \: : \: a \in A\}$ for the dilation of $A$ by a factor of $n$. For example, $n\bbS = [0,n]^2$, and $[n\bbS] = \{0,\dots, n\} \times \{0,\dots, n\}$.  

Recall that LQG is, heuristically, the metric with conformal factor the exponential of a GFF. We define DLFPP on $[n\bbS]$ by exponentiating a DGFF $\eta^n$ (see Section~\ref{subsection_DGFF} for the definition of a DGFF). 

\begin{definition}[Discrete Liouville first passage percolation distance]\label{def_model}
For $\xi > 0$ and $n \in \N$, consider a zero boundary DGFF $\eta^n$ on $[n\bbS]$. We define the \emph{($\xi$-)DLFPP distance} $D^\xi_{\eta^n}(u,v)$ between $u,v \in [n\bbS]$ to be zero if $u = v$, and otherwise the minimum of $\sum_{j=0}^k e^{\xi\sqrt{\pi/2} \eta^n(w_j)}$ over paths from $w_0 = u$ to $w_k = v$ in $[n\bbS]$ (equipped with its standard nearest-neighbor graph adjacency). 

Furthermore, for a vertex set $S \subset [n\bbS]$ and $u,v \in S$, we define the \emph{restricted DLFPP distance} $D^\xi_{\eta^n}(u,v;S)$ to be the above minimum taken over paths which stay in $S$. For subsets $A, B \subset S$, we define $D^\xi_{\eta^n}(A,B;S)$ to be the minimum of $D^\xi_{\eta^n}(a,b;S)$ for $a \in A, b \in B$. 
\end{definition}

We similarly define continuum LFPP by replacing the DGFF with the unit radius circle averages of a continuum GFF $h^n$ (see \cite[Section 3.1]{shef-kpz} for the definition of a GFF and its circle averages) and replacing lattice paths with piecewise continuously differentiable paths. 

\begin{definition}[Continuum Liouville first passage percolation distance]\label{def_LFPP}
For $\xi>0$ and $n \in \N$, consider a continuum zero boundary GFF $h^n$ on $n\bbS$ extended to zero outside $n\bbS$. The \emph{($\xi$-)continuum LFPP distance} $D^\xi_{h^n, \LFPP}(z,w)$ is the infimum over all piecewise continuously differentiable paths $P:[0,T] \to n\bbS$ from $z$ to $w$ of the quantity $\int_0^T e^{\xi h^n_1(P(t))} |P'(t)| \ dt$, where $h_1^n(z)$ denotes the average of $h^n$ over the radius 1 circle $\partial B_1(z)$.

For open $S \subset n\bbS$ and $z,w \in S$, we define the \emph{restricted continuum LFPP distance} $D^\xi_{h^n, \LFPP}(z, w ; S)$ to be the above infimum over paths in $\overline S$, and for subsets $A, B \subset S$ define $D^\xi_{h^n, \LFPP} (A,B;S) = \inf_{a\in A, b \in B} D^\xi_{h^n, \LFPP}(a,b;S)$. 
\end{definition}
\begin{remark}\label{rem_LFPP_defn}
Our definition of continuum LFPP differs slightly from that of other works \cite{ding-goswami-watabiki, dg-lqg-dim,  GP-LFPP}, which have an additional parameter controlling the circle average radius. In this work we will always take the circle average radius to be 1. 

We are interested in continuum LFPP distances in the domain $n\bbS$ with the circle average radius set to 1. Conversely, \cite[Theorem 1.5]{dg-lqg-dim} considers continuum LFPP distances in the domain $\bbS$ but with circle average radius $\delta$; we set $\delta = \frac1n$. By the scale invariance of the GFF, when we identify $n\bbS$ with $\bbS$ by a dilation and use the same GFF for both models, our continuum LFPP distances are exactly $n$ times larger than those of \cite{dg-lqg-dim} (this factor of $n$ arises from the rescaling of the paths $P$).
\end{remark}

The normalization factor $\sqrt{\frac\pi2}$ of Definition~\ref{def_model} arises because the GFF and DGFF we use have differing normalization constants. We expect that for fixed $\xi$, the above definitions of $\xi$-DLFPP and $\xi$-continuum LFPP have the same (conjectural) scaling limit as $n \to \infty$.

We come to our main theorem, that we can couple a GFF $h^n$ with a DGFF $\eta^n$ so that with high probability the circle average regularized GFF $h^n_1$ and the DGFF multiple $\sqrt{\frac\pi2} \eta^n$ are uniformly not too far apart. Under this coupling, with high probability $\xi$-DLFPP and $\xi$-continuum LFPP distances agree up to a factor of $n^{o(1)}$.

\begin{theorem}[Coupling of $\eta^n$ and $h^n$]\label{thm_coupling}
There exists a coupling of the GFF $h^n$ and the DGFF $\eta^n$ such that for each $\zeta > 0$ and open $U \subset \bbS$ with $\dit(U, \partial \bbS) >0$, with superpolynomially high probability as $n \to \infty$ we have 
\[\max_{v \in [nU]} \left| h^n_1(v) - \sqrt{\frac\pi2} \eta^n(v)\right| \leq \zeta \log n. \]
Under this coupling, for each $\zeta > 0$ and rectilinear polygon $P \subset \bbS$ with $\dit(P, \partial \bbS)>0$, with polynomially high probability as $n \to \infty$, we have uniformly for all $z, w \in nP$ that 
\[n^{-\zeta} D^\xi_{\eta^n} ([z],[w]; [nP]) \leq D^\xi_{h^n, \LFPP} (z, w; nP) \leq n^\zeta \left( D^\xi_{\eta^n} ([z],[w];[nP] ) + e^{\xi \sqrt{\frac\pi2}\eta^n([z])} \right). \]
\end{theorem}
The term $e^{\xi \sqrt{\frac\pi2} \eta^n([z])}$ in the upper bound takes care of the edge case where $[z] = [w]$ but $z \neq w$ (so $D^\xi_{\eta^n} ([z],[w];[nP]) = 0$ but $D^\xi_{h^n, \LFPP} (z,w;nP) > 0$). Also, the condition of $P$ being a rectilinear polygon can be weakened, but we prefer to avoid worrying about how the boundary $\partial P$ interacts with the lattice approximation.

Roughly speaking, Theorem~\ref{thm_coupling} tells us that DLFPP and continuum LFPP distances are close. Consequently, since \cite[Theorem 1.5]{dg-lqg-dim} gives the $\xi$-continuum LFPP distance exponent for $\xi = \frac\gamma{d_\gamma}$ in terms of the $\gamma$-LQG fractal dimension $d_\gamma$ (defined in \cite{dg-lqg-dim}), with a little effort we can obtain the same distance exponent for $\xi$-DLFPP, 
proving a conjecture of \cite[Section 1.5]{dg-lqg-dim}.

\begin{theorem}[DLFPP distance exponent]\label{thm_fractal_dimension}
Let $\gamma \in (0,2)$, and let $\xi = \frac{\gamma}{d_\gamma}$. Then for any distinct $z, w$ in the interior of $\bbS$, with probability tending to 1 as $n \to \infty$ we have
\eqb\label{eq_ptp}
D^{\xi}_{\eta^n} ([nz], [nw]) = n^{\frac{2}{d_\gamma} + \frac{\gamma^2}{2d_\gamma} + o(1)}.
\eqe
Furthermore, for any open $U \subset \bbS$ with $\dit(U, \partial \bbS) >0$ and compact $K \subset U$, with probability tending to 1 as $n \to \infty$ we have 
\eqb\nonumber
\max_{u,v \in [nK]} D^{\xi}_{\eta^n} (u,v; [nU]) = n^{\frac{2}{d_\gamma} + \frac{\gamma^2}{2d_\gamma} + o(1)} \quad \text{and} \quad D^{\xi }_{\eta^n} ([nK], [n\partial U]) =  n^{\frac{2}{d_\gamma} + \frac{\gamma^2}{2d_\gamma} + o(1)}.
\eqe
\end{theorem}

For small $\xi$, the paper \cite{DD19} establishes the existence of a subsequential $\xi$-DLFPP scaling limit as $n \to \infty$. Writing $\xi = \frac\gamma{d_\gamma}$, Theorem~\ref{thm_fractal_dimension} gives the exponential order of the distance normalization factors in terms of $d_\gamma$; namely, one should rescale distances by $n^{-\frac{2}{d_\gamma} - \frac{\gamma^2}{2d_\gamma} + o(1)}$. 

\cite[Theorem 1.2, Proposition 1.7]{dg-lqg-dim} tell us that $\gamma \mapsto \frac{\gamma}{d_\gamma}$ is continuous and increasing. Thus Theorem~\ref{thm_fractal_dimension} discusses the DLFPP distance exponent for $\xi \in (0,\frac2{d_2})$. To formulate things in full generality, we define the DLFPP distance exponent for all $\xi>0$.
\begin{definition}\label{def_distance_exponent}
Let $S_1 \subset S_2$ be squares with the same center as $\bbS$, and side lengths $\frac13$ and $\frac23$ respectively. For $\xi > 0$, define the \emph{$\xi$-DLFPP distance exponent} $\lambda(\xi)$ via
\[\lambda(\xi) = \sup \left\{ \alpha \: : \: \lim_{n \to \infty} \P\left[ D^\xi_{\eta^n}([nS_1], [n \partial S_2] ) < n^{1-\alpha} \right] = 1 \right\}.\]
\end{definition}
\begin{remark}\label{rem_lambda}
Our definition of $\lambda(\xi)$ is chosen to align with that of the $\xi$-continuum LFPP distance exponent defined in \cite[Equation (1.4)]{GP-LFPP}. By Theorem~\ref{thm_fractal_dimension} and \cite[Theorem 1.5]{dg-lqg-dim} these exponents agree for $\xi \in (0,\frac2{d_2})$. More strongly we expect that these two distance exponents agree for \emph{all} $\xi$, though this is not proved here. We note that the exponent bounds \cite[Theorem 2.3]{GP-LFPP} are applicable to our $\lambda(\xi)$; their proofs carry over without modification. 
\end{remark}

Using a result on the DGFF level set percolation (\cite{ding-li-chem-dist}, see also \cite{ding-wirth-level-perc}), we can easily establish a lower bound for $\lambda(\xi)$.
\begin{theorem}\label{thm_lambda_lower}
We have $\lambda(\xi) \geq 0$ for all $\xi >0$. 
\end{theorem}
\begin{proof}
Consider a DGFF $\eta^n$ on $[n\bbS]$, and fix any $\chi \in (\frac12, 1)$. Then \cite[Theorem 1]{ding-li-chem-dist} tells us that with probability tending to 1 as $n \to \infty$, there exists a path from $[nS_1]$ to $[n \partial  S_2]$ passing through at most $ne^{(\log n)^\chi}$ vertices, such that $\eta^n$ is at most $(\log n)^\chi$ uniformly along the path. As a result, with probability approaching 1 as $n \to \infty$, we have
\[D^\xi_{\eta^n}([nS_1], [n \partial S_2]) \leq ne^{(\log n)^\chi} e^{\xi \sqrt{\frac\pi2} (\log n)^\chi}= n^{1 + o(1)}.\]
\end{proof}
Theorem~\ref{thm_lambda_lower} is an improvement over previous lower bounds\footnote{See Remark~\ref{rem_lambda}.} for $\lambda(\xi)$ for the regime $\xi \in (\xi_0,0.266) \cup(0.708,\infty)$ (where $\xi_0>0$ is small and nonexplicit). Note that for $\xi \in (\xi_0, 0.241)$, the previous best lower bound \cite[Theorem 1.6]{ghs-map-dist} was obtained by working with mated-CRT maps and considering the ``LQG length'' of a deterministic Euclidean path via a KPZ relation \cite{shef-kpz}. Thus in some sense our result shows that for this range of $\xi$, deterministic Euclidean paths do not have low $\xi$-DLFPP lengths. 

For $\xi \in (0.267,0.707)$, stronger lower bounds were proved in \cite{dg-lqg-dim,GP-LFPP}, and for $\xi \in (0,\xi_0)$, \cite{ding-goswami-watabiki} gives $\lambda(\xi) \geq \Omega(\xi^{4/3} /\log (1/\xi))$. 

For $\xi = \frac{\gamma}{d_\gamma}$, Theorems~\ref{thm_fractal_dimension} and \ref{thm_lambda_lower} immediately yield the following lower bound for $d_\gamma$.

\begin{theorem} \label{thm_lower_bound_dgamma}
For $\gamma \in (0,2)$, the fractal dimension of $\gamma$-LQG $d_\gamma$ satisfies
\[d_\gamma \geq 2 + \frac{\gamma^2}{2}. \]
\end{theorem}
\begin{proof}
By Theorem~\ref{thm_fractal_dimension} we see that $\lambda(\gamma/d_\gamma) = 1 - \frac{1}{d_\gamma} \left( 2 + \frac{\gamma^2}2\right)$. Applying Theorem~\ref{thm_lambda_lower} yields the result. 
\end{proof}
By \cite[Theorems 1.4, 1.5, 1.6]{dg-lqg-dim}, this yields a bound for each of the $\gamma$-LQG discretizations discussed in \cite{dg-lqg-dim}, including the mated-CRT map, Liouville graph distance, and continuum LFPP.
As before, Theorem~\ref{thm_lower_bound_dgamma} is the best known lower bound for the regime $\gamma \in (\gamma_0,0.576)$, where $\gamma_0$ is small and non-explicit. The best known lower bounds for $\gamma \in (0.577, 2)$ are proved in \cite{dg-lqg-dim, GP-LFPP}, and the best bound for $\gamma \in (0,\gamma_0)$ is shown in \cite{ding-goswami-watabiki}.

Finally, we briefly comment on the Euclidean length exponent of the $\xi$-DLFPP annulus crossing geodesic (see Definition~\ref{def_distance_exponent} for the definition of the annulus).
\begin{remark}
For $\xi\in (0.267, 0.707)$ we have the bound $\lambda(\xi) > 0$ (\cite[Theorem 2.3]{GP-LFPP}, see Remark~\ref{rem_lambda}), so by \cite[Theorem 1.2, Remark 1.3]{dingzhang-geodesic} we see that with probability approaching 1 as $n \to \infty$, the (a.s. unique) annulus crossing DLFPP geodesic passes through at least $n^{1+\alpha}$ vertices for some $\alpha = \alpha(\xi) > 0$. That is, the Euclidean length exponent of the annulus crossing geodesic is strictly greater than 1. Also, for general $\xi$, \cite[Theorem 2.6]{GP-LFPP} gives an upper bound on the Euclidean length exponent of the DLFPP annulus crossing geodesic; their proof carries over to our setting with minor modification. 
\end{remark}

In Section~\ref{section_prelim}, we cover the necessary preliminaries. In Section~\ref{subsection_discrepancy} we prove the first part of Theorem~\ref{thm_coupling}, and in Section~\ref{subsection_DLFPP_vs_LFPP} we prove the second part of Theorem~\ref{thm_coupling}. Finally in Section~\ref{subsection_distance_exponent} we prove Theorem~\ref{thm_fractal_dimension}.
\section{Preliminaries}\label{section_prelim}
\subsection{Notation}\label{subsection-notation}
In this paper, we write $O(1)$ to denote some quantity that remains bounded as $n \to \infty$, and $o(1)$ for some quantity that goes to zero as $n \to \infty$. For any parameter $x$ we also write $O_x(1)$ to denote a quantity bounded in terms of $x$ as $n \to \infty$ while $x$ stays fixed. 

We say an event $A_n$ occurs with polynomially high probability as $n\to\infty$ if there is some positive constant $C>0$ such that $1-\P[A_n] \leq n^{-C}$ for all large $n$. Similarly, we say $A_n$ occurs with superpolynomially high probability if for all $C > 0$ we have $1-\P[A_n] \leq n^{-C}$ for sufficiently large $n$ in terms of $C$. 

\subsection{Discrete Gaussian free field}\label{subsection_DGFF}
For a set of vertices $V \subset \Z^2$, let $\partial V\subset V$ be the vertices having at least one neighbor outside $V$. The \emph{discrete Green function} $G^{V}(u,v)$ is the expected number of visits to $v$ of a simple random walk on $\Z^2$ started at $u\in V$ and killed upon reaching $\partial V$. 
The \emph{zero boundary DGFF} $\eta^n: [n\bbS] \to \R$ is a mean zero Gaussian process indexed by $[n\bbS]$ with covariances given by $\E[\eta^n(u)\eta^n(v)] = G^{[n\bbS]}(u,v)$. 
In particular, since $G^{[n\bbS]}(v,v) = 0$ whenever $v \in \partial [n\bbS]$, we have $\eta|_{\partial [n\bbS]} \equiv 0$. 

In our subsequent analysis, we will need the following DGFF local covariance estimate.
\begin{lemma}\label{lem_discrete_green_approx}
Let $U \subset \bbS$ be an open set satisfying $\dit(U, \partial \bbS)>0$. Then for fixed $k > 0$, for all $u,v \in [nU]$ with $|u - v|\leq k$ we have
\[ \E[\eta^n(u)\eta^n(v)]= \frac{2}{\pi}\log n + O_{U,k}(1),\]
where the term $O_{U,k}(1)$ is uniformly bounded for all $n,u, v$. 
\end{lemma}
\begin{proof}
By definition we need to show $ G^{[n\bbS]}(u,v) = \frac{2}{\pi}\log n + O_{U,k}(1)$. There exist $0<r<R$ such that for every point $z \in U$ we have $B_{r}(z) \subset \bbS \subset B_{R}(z)$, and by the domain monotonicity of the Green function we have $G^{[B_{nr}(u)]}(u,v) \leq G^{[n\bbS]} (u,v) \leq G^{[B_{nR}(u)]}(u,v)$. Using standard properties of the Green function (see, e.g., \cite[Theorem 4.4.4, Proposition 4.6.2]{lawler-limic-walks}), each of $G^{[B_{nr}(u)]}(u,v)$ and $G^{[B_{nR}(u)]}(u,v)$ is given by $\frac2\pi \log n + O_{r,R,k}(1)$, so we are done. 
\end{proof}

\subsection{DGFF as a projection}\label{subsection_proj}
The following lemma from \cite[Section 4.3]{shef-gff} relates the DGFF and continuum GFF. For $n \in \N$, the lattice $\Z^2$ divides the square $n \bbS$ into $n^2$ unit squares. Cutting each of these unit squares along its down-right diagonal gives us a triangulation of $n\bbS$; let $H^n$ be the (finite dimensional) space of continuous functions on $n\bbS$ which are affine on each triangle and vanish on $\partial (n\bbS)$. 
\begin{lemma} \label{lem_proj}
Suppose $h^n$ is a (continuum) zero boundary GFF on $\bbS$, and let $\sqrt{\frac\pi2} \eta^n$ be the projection of $h^n$ to $H^n$. Then $\eta^n$ restricted to $[n\bbS]$ has the law of a zero boundary DGFF on $[n\bbS]$. 
\end{lemma}
See \cite[Section 4.3]{shef-gff} for details on how to make sense of this projection. We note that the normalization constant $\sqrt{\frac\pi2}$ arises because our normalizations of the GFF and DGFF differ from those of \cite{shef-gff}.

Notice that given the values of $\eta^n$ restricted to $[n\bbS]$ and the fact that it is affine on each triangle, we can recover the function $\eta^n$ on the whole domain $n\bbS$ by linear interpolation within each triangle. Consequently, we will not distinguish between a DGFF defined on $[n\bbS]$ and a linearly-interpolated DGFF defined on $n\bbS$.

\begin{remark}\label{rem_independent}
With $h^n, \eta^n$ as in Lemma~\ref{lem_proj}, the distributions $\sqrt{\frac\pi2} \eta^n$ and $h^n - \sqrt{\frac\pi2} \eta^n$ are independent. This follows from the fact that the projections of $h^n$ to spaces orthogonal with respect to the Dirichlet inner product are independent; see \cite[Section 2.6]{shef-gff}.
\end{remark}

\section{Comparing DLFPP and continuum LFPP distances}\label{section_main}
In this section we prove Theorems~\ref{thm_coupling} and~\ref{thm_fractal_dimension}.

As in Lemma~\ref{lem_proj}, let $h^n$ be a zero boundary GFF on $\bbS$, and let $\sqrt{\frac\pi2} \eta^n$ be its projection onto $H^n$ (defined in Section~\ref{subsection_proj}). Recall that $\eta^n |_{[n\bbS]}$ has the law of a DGFF. 
Also write $h^n_1$ for the unit radius circle average regularization of $h^n$. 

In Section~\ref{subsection_discrepancy} we prove that away from the boundary, with superpolynomially high probability as $n \to \infty$ the discrepancy $|h^n_1 -\sqrt{\frac\pi2} \eta^n|$ is uniformly not too large. In Section~\ref{subsection_DLFPP_vs_LFPP}, we prove that as $n \to \infty$, with polynomially high probability  DLFPP and continuum LFPP distances differ by $n^{o(1)}$, proving Theorem~\ref{thm_coupling}. Finally in Section~\ref{subsection_distance_exponent} we use Theorem~\ref{thm_coupling} and the continuum LFPP distance exponent from \cite{dg-lqg-dim} to obtain the DLFPP distance exponent, proving Theorem~\ref{thm_fractal_dimension}.

\subsection{Discrepancy between DGFF and circle average regularized GFF} \label{subsection_discrepancy}
In this section, we establish that for an open set $U \subset \bbS$ with $\dit (U, \partial \bbS) > 0$, with high probability the discrepancy between $h^n_1$ and $\sqrt{\frac\pi2}\eta^n$ restricted to $[nU]$ is uniformly bounded by $o( \log n)$. 

We first show that the pointwise differences $h^n_1(v) - \sqrt{\frac\pi2} \eta^n(v)$ for $v \in [nU]$ have uniformly bounded variances. 

\begin{lemma}\label{lem_coupled_field_variance}
Let $h^n$ be a zero boundary GFF on $n\bbS$, and as in Lemma~\ref{lem_proj} let $\sqrt{\frac\pi2}\eta^n$ be its projection onto $H^n$. Then for any open $U \subset \bbS$ with $\dit(U, \partial \bbS) > 0$, uniformly for all $v \in [nU]$ we have
\[\Var \left(h^n_1(v) - \sqrt{\frac\pi2}\eta^n(v)\right) = O_U(1).\]
\end{lemma}

\begin{proof}
Observe that, writing $\eta^n_1$ for the unit radius circle average of the linearly interpolated DGFF $\eta^n$,
\eqb \label{eqn_circle_av_vs_dgff_var}
h^n_1(v) - \sqrt{\frac\pi2} \eta^n(v) = \left(h^n_1 (v) - \sqrt{\frac\pi2} \eta^n_1(v)\right) + \sqrt{\frac\pi2} (\eta^n_1(v) - \eta^n(v)).
\eqe
We will bound the variance of each of the two RHS terms by $O_U(1)$. By Remark~\ref{rem_independent}, we have $\Var(h^n_1 (v) - \eta^n_1(v)) = \Var h^n_1 (v) - \Var \eta^n_1(v)$. By \cite[Proposition 3.2]{shef-kpz} we have uniformly for $v \in [nU]$ that
\eqb \label{eqn_circle_av_var}
\Var h^n_1 (v) = \log n +O_U(1).
\eqe
We turn to analyzing $\Var \eta^n_1(v)$. Notice that since $\eta^n$ is affine on each triangle, we can write $\eta^n_1(v)$ as a weighted average of $\eta^n(u)$ for $u$ close to $v$. Concretely, let $N_v = \{ u \in \Z^2 \: : \: |u-v| < 2 \}$, then for deterministic nonnegative weights $\{w_u\}$ with $\sum_{u \in N_v} w_u = 1$ we have
\[
\eta^n_1 (v) = \sum_{u \in N_v} w_u \eta^n (u).
\]
Thus, by Lemma~\ref{lem_discrete_green_approx} we have
\eqb \label{eqn_dgff_var}
\Var \eta^n_1 (v) = \sum_{u,u' \in N_V}  w_uw_{u'}\frac2\pi \log n + O_U(1) = \frac2\pi\log n +O_U(1).  
\eqe
Combining \eqref{eqn_circle_av_var} and \eqref{eqn_dgff_var}, we conclude that $\Var(h^n_1 (v) - \sqrt{\frac\pi2} \eta^n_1(v)) = O_U(1)$, so we have bounded the variance of the first term of the RHS of \eqref{eqn_circle_av_vs_dgff_var}. We can bound the variance of the second term of \eqref{eqn_circle_av_vs_dgff_var} by $O_U(1)$ in exactly the same way that we derived \eqref{eqn_dgff_var}. We are done.
\end{proof}
\begin{remark}
By doing a more careful analysis of the discrete and continuum Green functions, one can improve the statement of Lemma~\ref{lem_coupled_field_variance} to the following: There exists some explicit universal constant $C$ such that for all open $U \subset \bbS$ with $\dit(U, \partial \bbS)>0$, for $n$ sufficiently large in terms of $U$, we have $\Var(h^n_1(v) - \sqrt{\frac\pi2}\eta^n(v)) < C$ for all $v \in [nU]$. This statement is unnecessary for our purposes so we omit its proof.
\end{remark}

Since $\#[n U] \leq n^2$ is not too large, we can show using Lemma~\ref{lem_coupled_field_variance} that with high probability the GFF circle-average field and the DGFF are uniformly not too different for all $v \in [nU]$.

\begin{lemma} \label{lem_coupling_close}
For $n>1$ and $U, h^n$ and $\eta^n$ as in Lemma~\ref{eqn_circle_av_vs_dgff_var}, there is a constant $C$ depending only on $U$ so that 
\[\P \left[\max_{v\in[nU]} \left(h^n_{1}(v) - \sqrt{\frac\pi2}\eta^n(v)\right) \geq C \sqrt{ \log n} +x \right] \leq e^{-x^2/2C} .\]
\end{lemma}
\begin{proof}
For notational convenience write $\Delta^n(v) = h^n_{1}(v) - \sqrt{\frac\pi2} \eta^n(v)$; this is a centered Gaussian random variable. Let $\wt C$ (depending only on $U$) be an upper bound for $\Var( \Delta^n(v))$ for all $v \in [nU]$ (Lemma~\ref{lem_coupled_field_variance}). Then by a standard Gaussian tail bound we have for any $v \in [nU]$ and $r>0$ that
\alb
\E[\Delta^n(v) \mathbbm1\{ \Delta^n(v) \geq r\} ] 
\leq  \sqrt{\frac{\wt C}{2\pi}} \left( 1 + \frac{\wt C}{r^2} \right) e^{-r^2/2\wt C} . 
\ale
Consequently we can set $r = \sqrt{2 \wt C \log n^{2}}$ and take a union bound to deduce, for some $C > \wt C$ depending only on $U$,
\[\E[\max_{v \in [nU]} \Delta^n(v)] \leq  r + \sum_{v \in [nU]} \E[\Delta^n(v) \mathbbm1\{ \Delta^n(v) \geq r\} ] \leq  C \sqrt{ \log n}. \]
Finally, since $C > \Var (\Delta^n(v))$ for all $v \in [nU]$, we can apply the Gaussian concentration inequality (see for instance \cite[Theorem 7.1]{ledoux-concentration}) to obtain Lemma~\ref{lem_coupling_close}.
\end{proof}
As an immediate corollary, by setting $x\asymp \log  n$ we deduce the following, which is the first part of Theorem~\ref{thm_coupling}.
\begin{proposition}\label{prop_coupling_close}
For $U, h^n$ and $\eta^n$ as in Lemma~\ref{eqn_circle_av_vs_dgff_var}, for any fixed $\zeta > 0$ we have with superpolynomially high probability as $n \to \infty$ that 
\[\max_{v \in [nU]} \left|h^n_1(v) - \sqrt{\frac\pi2}\eta^n(v)\right| \leq \zeta  \log n. \]
\end{proposition}

\subsection{Comparing DLFPP to continuum LFPP}\label{subsection_DLFPP_vs_LFPP}
In this section, we use the comparison result Proposition~\ref{prop_coupling_close} and \cite[Proposition 3.16]{dg-lqg-dim} to prove the second part of Theorem~\ref{thm_coupling}: under the coupling of Proposition~\ref{prop_coupling_close}, with high probability DLFPP and continuum LFPP distances differ by a multiplicative factor of $n^{o(1)}$. 

Recall that we defined the $\xi$-DLFPP distance $D^\xi_{\eta^n}$ and $\xi$-continuum LFPP distance $D^\xi_{h^n, \LFPP}$. We further define the lattice LFPP distance $D^{\xi, \lattice}_{h^n, \LFPP}$ in exactly the same way that we define $D^\xi_{\eta^n}$ in Definition~\ref{def_model}, except we use vertex weights of $e^{\xi h^n_1(v)}$ rather than $e^{\xi \sqrt{\pi/2} \eta^n(v)}$. 

In Lemma~\ref{lem_circ_av_compr_to_LFPP}, using \cite[Proposition 3.16]{dg-lqg-dim} we check that with high probability $D^\xi_{h^n, \LFPP}$ and $D^{\xi, \lattice}_{h^n, \LFPP}$ are comparable. Since Proposition~\ref{prop_coupling_close} tells us that $D^{\xi, \lattice}_{h^n, \LFPP}$ and $D^\xi_{\eta^n}$ are comparable, we complete the proof of Theorem~\ref{thm_coupling}.

\begin{lemma}\label{lem_circ_av_compr_to_LFPP}
For each $\xi, \zeta > 0$ and open rectilinear polygon $P \subset \bbS$ with $\dit(P, \partial \bbS) >0$, with polynomially high probability as $n\to \infty$ we have for all $z,w \in nP$ that 
\eqb\label{eq_circ_av_compr_to_LFPP}
n^{-\zeta} D^{\xi, \lattice}_{h^n, \LFPP}([z],[w];[nP])
\leq D^\xi_{h^n, \LFPP}(z,w;nP) \leq n^\zeta \left( D^{\xi, \lattice}_{h^n, \LFPP} ([z],[w];[nP]) + e^{\xi h^n_1([z])} \right) .
\eqe
\end{lemma}
\begin{proof}
This is precisely the statement of \cite[Proposition 3.16]{dg-lqg-dim}, but with three differences which we address in turn. 
\begin{itemize}
\item It considers LFPP in a fixed domain, but sends the circle average radius $\delta$ to zero. This is in contrast with our setting where we have LFPP in a growing domain but fix the circle average radius.
\medskip

\noindent
This difference is cosmetic; see Remark~\ref{rem_LFPP_defn}. We set $\delta = \frac1n$ and then scale everything in \cite[Proposition 3.16]{dg-lqg-dim} up by a factor of $n$ so that it discusses LFPP in $n\bbS$ with unit radius circle averages. Henceforth we consider the scaled-up version of \cite[Proposition 3.16]{dg-lqg-dim}.

\item It uses $[-1,2]^2$ and $\bbS$ rather than our sets $\bbS$ and $P$ respectively.
\medskip

\noindent
The same method of proof applies, since the proofs of their Lemmas 3.4 and 3.7 require only $\dit(\bbS, \partial [-1,2]^2) > 0$ (we assume the corresponding $\dit(P, \partial \bbS) > 0$), and their argument for replacing curves in $n\bbS$ with lattice paths in $[n\bbS]$ (and vice versa) works when one replaces $\bbS$ with $P$, for sufficiently large $n$.

\item In our rescaled notation, instead of proving \eqref{eq_circ_av_compr_to_LFPP}, \cite[Proposition 3.16]{dg-lqg-dim} instead proves 
\eqb\label{eq_3.16}
n^{-\zeta} \left( \wh D^{\delta}_{\LFPP}([z],[w];[nP]) - e^{\xi \wh h_\delta([z])} \right)
\leq D^\xi_{h^n, \LFPP}(z,w;nP) \leq n^\zeta \wh D^{\delta}_{\LFPP}([z],[w];[nP]),
\eqe
where $\wh h_\delta$ is a field coupled to $h^n$, and $\wh D^\delta_{\LFPP}$ is defined the same way as $D^{\xi, \lattice}_{h^n, \LFPP}$ except we use vertex weights of $e^{\xi \wh h_\delta}$ instead of $e^{\xi h^n_1}$, and also set for all $v \in [nP]$ that $\wh D^\delta_{\LFPP}(v,v;[nP]) = e^{\xi \wh h_\delta(v)}$ instead of 0. 
\medskip

\noindent
Firstly, we modify the definition of $\wh D^\delta_{\LFPP}(v,v;[nP])$, setting it equal to 0 instead of $ e^{\xi \wh h_\delta(v)}$, and correspondingly move the correction term $e^{\xi \wh h_\delta([z])}$ from the lower bound to the upper bound in \eqref{eq_3.16}. Next, \cite[Equation (3.34)]{dg-lqg-dim} tells us that with polynomially high probability as $n \to \infty$, for all $z\in nP$ we have $|\wh h_\delta (z) - h^n_1(z)| \leq \zeta \log n$. Thus, with polynomially high probability as $n \to \infty$, we can replace $\wh h_\delta$ and $D^{\xi, \lattice}_{h^n, \LFPP}$ with $h^n_1$ and $\wh D^\delta_{h^n, \LFPP}$ in \eqref{eq_3.16}, incurring a factor of $n^{\zeta}$. This gives \eqref{eq_circ_av_compr_to_LFPP} with $2\zeta$ instead of $\zeta$, so we are done. 
 \end{itemize}
\end{proof}

Using Lemma~\ref{lem_circ_av_compr_to_LFPP} and Proposition~\ref{prop_coupling_close}, we prove Theorem~\ref{thm_coupling}. 

\begin{proof}[Proof of Theorem~\ref{thm_coupling}]
The first assertion of Theorem~\ref{thm_coupling} is Proposition~\ref{prop_coupling_close}. For the second, notice that Proposition~\ref{prop_coupling_close} implies that with polynomially high probability as $n \to \infty$, uniformly over $z, w \in nP$ we have $h^n_1([z]) \leq \sqrt{\frac\pi2} \eta^n([z]) + \zeta \log n$ and 
\[n^{-\zeta} D^\xi_{\eta^n}([z],[w];[nP]) \leq  D^{\xi, \lattice}_{h^n,\LFPP}([z],[w];[nP]) \leq  n^{\zeta} D^\xi_{\eta^n}([z],[w];[nP]). \]
Combining this with Lemma~\ref{lem_circ_av_compr_to_LFPP} yields the second assertion of Theorem~\ref{thm_coupling} (with $2\zeta$ instead of $\zeta$).
\end{proof}

\subsection{The DLFPP distance exponent}\label{subsection_distance_exponent}
Finally, we use Theorem~\ref{thm_coupling} with the continuum LFPP distance exponent from \cite{dg-lqg-dim} to obtain Theorem~\ref{thm_fractal_dimension}.

\begin{lemma}\label{lem_apprx_bounds}
For $\gamma \in (0,2)$, set $\xi = \frac{\gamma}{d_\gamma}$. Let $U \subset \bbS$ be an open set with $\dit(U, \partial \bbS) > 0$, and $K \subset U$ a compact set. Then with probability tending to 1 as $n \to \infty$ we have  
\eqb\label{eq_apprx_upper}
\max_{u,v \in [nK]} D^{\xi}_{\eta^n }(u,v;[nU]) \leq n^{\frac{2}{d_\gamma} +\frac{\gamma^2}{2d_\gamma} + o(1)}.
\eqe
and
\eqb\label{eq_apprx_lower}
D^{\xi}_{\eta^n }([nK], [n\partial U]) \geq n^{\frac{2}{d_\gamma} +\frac{\gamma^2}{2d_\gamma} -o(1)}
\eqe
\end{lemma}
\begin{proof}
We use the coupling of Theorem~\ref{thm_coupling}. By \cite[Lemma 2.1, Theorem 1.5]{dg-lqg-dim} and Remark~\ref{rem_LFPP_defn}, we see that for any $\zeta > 0$, with probability tending to 1 as $n \to \infty$ we have
\eqb\label{eq_LFPP_bds}
\max_{z,w \in nK} D^\xi_{h^n, \LFPP} (z,w;nU) \leq n^{\frac{2}{d_\gamma} +\frac{\gamma^2}{2d_\gamma}+\zeta}.
 \quad \text{ and } \quad
 D^\xi_{h^n, \LFPP}(nK, n\partial U) \geq n^{\frac{2}{d_\gamma} + \frac{\gamma^2}{2d_\gamma}-\zeta}    
\eqe

We first prove \eqref{eq_apprx_upper}. Choose open rectilinear $P$ so that $K \subset P \subset U$ and $\dit(K, \partial P ), \dit(P, \partial U) > 0$. Clearly $D^{\xi}_{\eta^n}(u,v;[nU]) \leq D^{\xi}_{\eta^n}(u,v;[nP])$ for any $u,v \in [nK]$. 
Thus, combining the lower bound of Theorem~\ref{thm_coupling} with the first equation of \eqref{eq_LFPP_bds} (with $U$ replaced by $P$) gives \eqref{eq_apprx_upper}:
\[\max_{u,v \in [nK]} D^{\xi}_{\eta^n}(u,v;[nU]) \leq \max_{u,v \in [nK]}  D^{\xi}_{\eta^n}(u,v;[nP]) \leq n^\zeta \max_{z,w \in nK} D^\xi_{h^n, \LFPP} (z,w;nP) \leq n^{\frac{2}{d_\gamma} +\frac{\gamma^2}{2d_\gamma}+2\zeta}. \]

Next we prove \eqref{eq_apprx_lower}. Choose any rectilinear polygon $P$ with $U \subset P \subset \bbS$ and $\dit(U,\partial P),\dit(P,\partial \bbS) > 0$. Note that any shortest path from $K$ to $\partial U$ stays in $P$. Combining the upper bound of Theorem~\ref{thm_coupling} with the second equation of \eqref{eq_LFPP_bds} gives, with probability tending to 1 as $n \to \infty$,
\eqb \label{eq_dominating_term}
D^{\xi}_{\eta^n }([nK], [n\partial U]) \geq n^{\frac{2}{d_\gamma} +\frac{\gamma^2}{2d_\gamma} -\zeta} - e^{\xi \sqrt{\frac\pi2} \max_{v} \eta^n(v)}.
\eqe
We now check that the first term in the RHS dominates the second, and consequently obtain \eqref{eq_apprx_lower}.
\cite[Theorem 2]{bolthausen2001} states that the maximum of a zero boundary DGFF on $[n\bbS]$ is $(1+o(1))2\sqrt{\frac2\pi}  \log n$ with probability tending to 1 as $n \to \infty$. Thus with probability tending to 1 as $n \to \infty$ we have $\max_{v \in [n\bbS]} \sqrt{\frac\pi2}\eta^n(v) \leq (2 + \zeta) \log n $, and since $\frac2{d_\gamma} + \frac{\gamma^2}{2d_\gamma} > \frac{ 2\gamma}{ d_\gamma} = 2\xi  $ for $\gamma \in (0,2)$, we conclude that for $\zeta>0$ sufficiently small in terms of $\gamma$ we have
\[ D^{\xi}_{\eta^n }([nK], [n\partial U]) \geq n^{\frac{2}{d_\gamma} +\frac{\gamma^2}{2d_\gamma} -\zeta} - n^{2\xi + \xi \zeta} = (1-o(1) )n^{\frac{2}{d_\gamma} +\frac{\gamma^2}{2d_\gamma} -\zeta}  .
\]
Hence we have \eqref{eq_apprx_lower}, so we are done.
\end{proof}
Now we are ready to prove Theorem~\ref{thm_fractal_dimension}.

\begin{proof}[Proof of Theorem~\ref{thm_fractal_dimension}]
To compute the point-to-point distance \eqref{eq_ptp}, first apply \eqref{eq_apprx_lower} with $K, U$ chosen so $z \in K$ and $w \not \in U$ to get the lower bound, then apply \eqref{eq_apprx_upper} with $K$ containing both $z$ and $w$ to get the upper bound.

Finally, \eqref{eq_ptp} with $z \in K$ and $w \not \in U$ gives us the the upper bound for $D^\xi_{\eta^n}([nK], [n\partial U])$, and \eqref{eq_ptp} with any distinct $z,w \in K$ yields the lower bound for $\max_{u,v \in [nK]}D^\xi_{h^n}(u,v; [nU])$.
\end{proof}

\bibliographystyle{hmralphaabbrv}
\bibliography{cibib,bibmore}

\end{document}